 \documentclass{amsart}

\usepackage{mathptmx}       %
\usepackage{helvet}         %
\usepackage{courier}        %
\usepackage{type1cm}        %
\usepackage{graphicx}        %
\usepackage{multicol}        %
\usepackage[bottom]{footmisc}%

\usepackage{amsmath}
\usepackage{amsfonts}
\usepackage{amssymb}
\usepackage{amsthm}

\newtheorem{theorem}{Theorem}

\theoremstyle{definition}
\newtheorem{question}[theorem]{Question}

\usepackage[utf8]{inputenc} 

\usepackage{color}
\usepackage{hyperref}
\usepackage{enumitem}

\def\cR{\mathcal{R}}

\def\QQ{\mathbb{Q}}
\def\RR{\mathbb{R}}

\def\RR{\mathbb{R}}
\def\SS{\mathbb{S}}

\def\Z{\mathbb{Z}}
\DeclareMathOperator{\conv}{conv}

\newcommand{\Mob}[1]{\operatorname{M\ddot{o}b}({#1})} %
\newcommand{\rspol}{\cR_{\text{pol}}}%
\newcommand{\rsins}{\cR_{\text{ins}}}%
\newcommand{\Sp}[1]{\SS^{#1}}%
\newcommand{\defn}[1]{\emph{#1}} 
\newcommand{\set}[2]{\ensuremath{\left\{#1\,\middle|\,#2\right\}}} %
\newcommand{\floor}[1]{\left\lfloor {#1} \right\rfloor} %
\newcommand{\ffloor}[2]{\left\lfloor{\frac{#1}{#2}}\right\rfloor} %
\newcommand{\bigslant}[2]{{\raisebox{.3em}{$#1$} \Big/ \raisebox{-.3em}{$#2$}}}

\makeindex             %

\begin{document}

\title{Six topics on inscribable polytopes
}
\author{Arnau Padrol}
\address{
 IMJ-PRG, UPMC (Paris 6), Case 247, 4 Place Jussieu, 75252 Paris Cedex 05, France}
\email{arnau.padrol@imj-prg.fr}
\author{G\"unter M. Ziegler}
\address{Institut f\"ur Mathematik, FU Berlin, Arnimallee 2, 14195 Berlin, Germany}
\email{ziegler@math.fu-berlin.de}

\thanks{This research was supported by the DFG Collaborative Research Center TRR~109 ``Discretization in Geometry and Dynamics''.}
\begin{abstract}
Inscribability of polytopes is a classic subject but also a lively research area nowadays.
We illustrate this with a selection of well-known results and recent developments 
on six particular topics related to inscribable polytopes. Along the way we collect a list of (new and old) open questions.
\end{abstract}

\maketitle

Jakob Steiner ended 
his 1832 geometry book \emph{Systematische Ent\-wicklung der Ab\-h\"an\-gig\-keit geometrischer Gestalten von einander}~\cite{Steiner1832} with a list of 85 open problems. Problem 77 reads as follows:

\begin{quotation}
77) Wenn irgend ein convexes Polyeder gegeben ist l\"a{\ss}t sich dann immer (oder in welchen F\"allen nur) irgend ein anderes, welches mit ihm in Hinsicht der Art und der Zusammensetzung der Grenzfl\"achen \"ubereinstimmt (oder von gleicher Gattung ist), in oder um eine Kugelfl\"ache, oder in oder um irgend eine andere Fl\"ache zweiten Grades beschreiben (d.h.\ da{\ss} seine Ecken alle in dieser Fl\"ache liegen oder seine Grenzfl\"achen alle diese Fl\"ache ber\"uhren)?
\end{quotation}

It asks whether every ($3$-dimensional) polytope is \defn{inscribable}; that is, whether for every $3$-polytope there is a combinatorially equivalent polytope 
with all the vertices on the sphere. And if not, which are the cases of $3$-polytopes that do have such a realization? He also asks the same question for \defn{circumscribable} polytopes, those that have a realization with all the facets tangent to the sphere, as well as for other surfaces of degree~$2$.

There was no progress on this question until 1928, when Ernst Steinitz showed that 
inscribability and circumscribability are polar concepts and presented
the first examples of polytopes that cannot be inscribed/circumscribed.
Since then, the interest 
on inscribability of polytopes has soared, partially because of tight relations
with Delaunay subdivisions and hyperbolic geometry. 
It was in the context of the latter that, more than 50 years after Steinitz's results, 
Igor Rivin finally found a characterization of $3$-dimensional inscribable
polytopes~\cite{Rivin1996}: A $3$-connected planar graph describes an inscribable polytope 
if and only if a certain inequality system has a solution.

What about other quadric surfaces? First of all, since these are not necessarily convex, 
one has to decide to either consider realizations with the vertices on the surface, or
realizations whose intersection with the surface are only the vertices. 
The weakly inscribable spherical polytopes considered in~\cite{ChenPadrol2015} belong to the first category.
For the second version of the definition, Jeffrey Danciger, Sara Maloni and Jean-Marc Schlenker have 
very recently extended Rivin's results to arbitrary quadrics in~$\RR^3$~\cite{DancigerMaloniSchlenker2015}: a $3$-polytope is inscribable in the hyperboloid or the cylinder if and only if it is inscribable in the sphere and its graph is Hamiltonian.

So, is the inscribability problem completely solved? We do not think so: Many fundamental questions in this area remain still wide open. In particular,
very little is known about inscribability for higher dimensional polytopes.
Here we present some intriguing open questions and problems motivated by some recent (and not so recent) results on inscribable polytopes.

\section{Inscribability of 3-polytopes}

Incribability and circumscribability are polar concepts: A polytope is inscribable if and only if its polar is circumscribable. Steinitz~\cite{Steinitz1928} (cf.~\cite[Thm.~13.5.2]{Gruenbaum2003})
constructed non-circumscribable polytopes using the following simple fact. 
Paint some of the facets of a $3$-polytope~$P$ black in such a way that
there are no two neighboring black facets (just like on a soccer ball). If you can paint more than half of the facets black, then $P$ is not circumscribable (unlike the soccer ball). 

His argument was the following. Observe that
each facet $F$ has a point of contact with the sphere $p_F$. We associate to each edge $e$ of $F$ the angle with which we see $e$ from $p_F$. A reflection about the plane spanned by $e$ and the center of the sphere shows that this angle is
the same for the two facets incident to $e$.
Now, if we add all these angles for the edges incident to black facets, we get $2\pi$ for each black facet. No two of these share an edge, so if we count the contributions for the white facets we should get at least the same value. However, we get at most $2\pi$ times the number of white facets, which is smaller than the number of black facets by hypothesis.\hspace*{\fill}\qed

The same argument also works if exactly half of the facets are black as long as there is at least one edge between two white facets. This provides us with our first example of a
non-circumscribable polytope: Take a simplex and truncate all its vertices.  The facets arising from the truncation do not share an edge, so painting them black shows non-circumscribability.

The polar argument says that if the graph of a $3$-polytope has an independent set (a subset of vertices no two of which are connected by an edge) of more than half of the vertices (or exactly half of them but not incident to every edge), then the polytope is not inscribable. For example, the \defn{triakis tetrahedron}, a convex polytope obtained by stacking a tetrahedron onto each facet of a tetrahedron, is not inscribable.

Steinitz's condition was later subsumed by results by Michael B. Dillencourt~\cite{Dillencourt1990} and by Dillencourt together with Warren D. Smith~\cite{DillencourtSmith1995}. In~\cite{Dillencourt1990} it is shown that the graph of any
inscribed polytope is \defn{$1$-tough}, which means that for any~$k$, removing $k$ vertices splits the graph into at most $k$ connected components.

This proves non-inscribability for polytopes with an independent set that contains more than half of the vertices,  by removing the other vertices. For independent sets that collect exactly half the vertices, we need a slight improvement from~\cite{DillencourtSmith1995}: the graph of an inscribable $3$-polytope is either bipartite (with both sides of the same size)
 or $1$-supertough, which means that for any $k\geq 2$, the removal of $k$ vertices splits the graph into less than $k$ components.

Toughness is a necessary combinatorial condition for inscribability that is easy to check, but it is not sufficient. 

Besides the mentioned necessary conditions, Dillencourt and
Smith also found some sufficient combinatorial conditions~\cite{DillencourtSmith1996}, which they summarize as: 
\emph{``If
a polyhedron has a sufficiently rich collection of Hamiltonian subgraphs, then it is of inscribable type.''} 
For example, this implies that
$3$-polytopes whose graphs are \mbox{$4$-connected}, or where
all the vertex degrees are between $4$ and $6$, are always inscribable.

So, how can we decide whether a given polytope is incribable?
Fundamental results on hyperbolic polyhedra by Rivin~\cite{Rivin1996} provided
an easy way to decide inscribability and circumscribability of $3$-polytopes~\cite{HodgsonRivinSmith1992} by linear programming 
(see also~\cite{Rivin1993}, \cite{Rivin1994}, \cite{Rivin2003}). If one identifies the
ball with the Klein model of the hyperbolic space, then an inscribed polytope is an
ideal hyperbolic polyhedron. Rivin showed that the dihedral angles at the edges 
completely characterize
these polyhedra.

\begin{theorem}[{\cite[Thm.~1]{HodgsonRivinSmith1992}}]
 A $3$-polytope $P$ is circumscribable if and only if there exist numbers 
 $\omega(e)$ associated to the edges $e$ of $P$ such that:
 \begin{itemize}
  \item[$\bullet$ ] $0<\omega(e)<\pi$,
  \item[$\bullet$ ] $\sum_{e\in F} \omega(e) =2\pi$ for each facet $F$ of $P$, and
  \item[$\bullet$ ] $\sum_{e\in C} \omega(e) >2\pi$ for each simple circuit $C$ that does not 
  bound a facet.
 \end{itemize}
\end{theorem}

Using this result, inscribability and circumscribability can be efficiently checked.
Yet, the characterization depends on a linear programming type feasibility computation.
As Dillencourt and Smith pointed out in~\cite{DillencourtSmith1995},
it is an outstanding open problem to find a graph-theoretical characterization
of inscribable $3$-polytopes. 
For simple polytopes, they found such a characterization~\cite{DillencourtSmith1995}: A simple $3$-polytope is inscribable if and only if its graph is either bipartite and has a $4$-connected dual or it is $1$-supertough.

\begin{question}[Dillencourt and Smith~\cite{DillencourtSmith1995}]
 Is there a purely combinatorial characterization of the graphs of inscribable $3$-polytopes?
\end{question}

\section{A characterization in higher dimensions}

It is not likely that there is a characterization as nice and simple as Rivin's for 
higer-dimensional inscribable polytopes. To start with,
$4$-dimensional polytopes already present \defn{universality} in the sense of Mn\"ev:
This is a very strong statement whose history took off with groundbreaking results in Nikolai Mn\"ev's PhD 
thesis~\cite{Mnev1986} \cite{Mnev1988}, and that we discuss a little further in Section~\ref{sec:universality}
(see also \cite{RichterGebert1997}).
It has several implications, among them that it is already very hard to decide whether a given face lattice corresponds to a $4$-polytope.
And as we will see later, higher-dimensional inscribable polytopes
also present universality features, so one should not expect to be able to 
decide inscribability in higher dimensions
easily or quickly.

A more realistic goal would hence be to look for strong necessary conditions 
for inscribability in higher dimensions 
as well as for good (that is, weak) sufficient conditions. 
A set of necessary conditions is available, since 
Steinitz's proof carries over directly to higher dimensions, as Branko Gr\"unbaum and 
Ernest Jucovi\v{c} have observed already in 1974~\cite{GruenbaumJucovic1974}. Let's see how.

\begin{theorem}%
 Let $P$ be a $d$-polytope with graph $G$. If $G$ has an independent set 
 that contains more than half of all the vertices, 
 or exactly half the vertices when $G$ is not bipartite,
 then $P$ is not inscribable.
\end{theorem}

\begin{proof}
Again, we prove the polar statement. Assume that $P^\circ$ is circumscribed. Each facet~$F$
touches the ball in a single point~$p_F$. To each of the facets of $F$, which are ridges of $P^\circ$, 
we can intersect the cone
with apex $p_F$ spanned by the ridge with a small ball centered at $p_F$. The (normalized) solid angle associated to the ridge is the ratio of the volume of this intersection to the volume of the ball. Again, a reflection shows that the solid angle associated to each ridge does not depend on which of the two incident facets we take the apex from.

Now assume that the facets of $P^\circ$ are painted in black and white, in such a way that there are no two neighboring black facets. If we add the contributions of the angles of the ridges associated to black facets, it should be at most (or less than if the dual graph is not bipartite) the sum of the contributions for the white facets. 
Since the sum along each 
facet is~$1$, the number of black facets cannot exceed the number of white facets, and they
also cannot be the same if there are two adjacent white facets. %
\end{proof}

However, the more general necessary conditions by Dillencourt, such as $1$-toughness, have not yet been generalized to higher dimensions, as far as we know; neither
have the sufficient conditions of Dillencourt and Smith. Is it true that if the graph of a
$d$-polytope has a rich enough structure of Hamiltonian subgraphs then the corresponding polytope is inscribable? This could explain the phenomena observed in the next section.

\begin{question}
 Find strong necessary and sufficient conditions for inscribability of higher-dimensional polytopes.
\end{question}

Moritz Firsching recently achieved a complete enumeration of the simplicial $4$-polytopes with $10$ vertices~\cite{Firsching2015}: There are exactly $162\,004$ combinatorial types.
Firsching also managed to decide inscribability for all but 13 of these polytopes: The remaining
$161\,991$ are divided into $161\,978$ inscribable and $13$ non-inscriable.
So, most simplicial $4$-polytopes with $10$ vertices are inscribable. 
Moreover, the very few that are not inscribable have very few facets and edges.
More precisely, every non-inscribable simplicial $4$-polytope with $10$ vertices
has less than $27$ facets and less than $37$ edges, but $146\ 104$ out of the $162\ 004$
simplicial $4$-polytopes with $10$ vertices do have at least $27$ facets and $37$ edges.
This seems indicate that there might be combinatorial sufficient conditions for
inscribability based only on the $f$-vector.

\begin{question}
 Is it true that most simplicial $4$-polytopes on $n$ vertices are inscribable, for $n\rightarrow \infty$?
\end{question}

In 1991 Smith \cite{Smith1991} proved that although there are exponentially many inscribable and circumscribable $3$-polytopes with $n$ vertices (because there are many that are $4$-connected), most simplicial $3$-polytopes with $n$ vertices are neither inscribable nor circumscribable. The proof consists in showing that the probability of finding a fixed non-inscribable subgraph in a random $3$-connected planar triangulations tends to $1$ as $n\to \infty$. 
The same argument can be used on random $3$-connected planar graphs, see~\cite[Thm.~2 and Cor.~1]{BenderGaoRichmond1992}, which shows that most combinatorial types of $3$-polytopes on a given large number of edges are not inscribable/circumscribable.
This contrasts with what happens with simplicial $3$-polytopes with few vertices. In the same paper~ \cite{Smith1991}, Smith classified these according to their inscribability and circumscribility, and most turned out to be both inscribable and circumscribable. 
The referee for this paper suggested to use the strategy for $3$-polytopes in higher dimension: To show that a large random simplicial $d$-polytope is likely to contain a fixed non-inscribable subcomplex.

\section{Neighborly polytopes}

What is the maximal number of faces that a $d$-dimensional polytope with $n$ vertices can have? The answer to this fundamental question in the combinatorial theory of
polytopes was not established until 1970, by Peter McMullen \cite{McMullen1970},
although Theodore Motzkin had already guessed the answer, 
as we know from a 1957 abstract~\cite{Motzkin1957}. And the
answer is that the \defn{cyclic $d$-polytopes with $n$ vertices} have the maximal number of $k$-faces among all $d$-polytopes with $n$ vertices (for all $k$!). This is the polytope one obtains
when taking the convex hull of $n$ points on the moment curve given by $\gamma(t):=(t,t^2, t^3,\dots, t^d)$.

The cyclic polytopes owe their name to David Gale~\cite{Gale1963}, one of several (re-)\allowbreak inventors
(cf.~\cite[Sec.~7.4]{Gruenbaum2003}), although they were essentially
already known to Constantin Carath\'eodory in 1911~\cite{Caratheodory1911}, 
who had studied the convex hull of the \defn{trigonometric moment curve} 
$\tau(t)= (\sin t, \cos t, \sin 2t, \cos 2t, \dots, \sin kt, \cos kt)$. The
representation on the trigonometric moment curve, which is projectively equivalent to the monomial one \cite[pp.~75/76]{Ziegler1995},
shows
that in even dimensions, the cyclic polytopes are inscribable. Since then, several inscribed realizations of the cyclic polytope 
have been found (also in odd dimensions), see \cite[p.~67]{Gruenbaum2003}, \cite[p.~521]{Seidel1991} and~\cite[Prop.~17]{GonskaZiegler2013}. 
Thus the upper bound theorem is established for inscribed polytopes too.

The next natural question asks for \emph{all} the polytopes that have this many facets. 
McMullen also provided the answer to this: This characterizes the simplicial \defn{neighborly polytopes}. 
These are the simplicial polytopes with a complete $\ffloor{d}{2}$-skeleton. 
(All even-dimensional neighborly polytopes are simplicial, but not in general the odd-\allowbreak dimen\-sional ones. In particular, not all odd-dimensional neighborly polytopes have the maximal number of facets, only those that are simplicial.)
Even if Motzkin claimed that the cyclic polytopes are the only neighborly polytopes 
(in the same 1957 abstract~\cite{Motzkin1957} mentioned above), there are actually plenty.
The number of neighborly $d$-polytopes with $n$ vertices grows at least as $n^{\floor{d/2}n(1-o(1))}$ for fixed~$d$~\cite{Padrol2013}. Compare this to the number of $d$-polytopes with $n$ vertices, which is not larger than $n^{d^2n(1+o(1))}$~\cite{Alon1986}. As it turns out, each one of the neighborly polytopes used to provide this lower bound is also inscribable~\cite{GonskaPadrol2015} (for $d\geq4$). And even more, as we discuss in the next section, they are inscribable in any strictly convex body! This surprising behavior let Gonska and Padrol \cite{GonskaPadrol2015} to ask:

\begin{question}[Gonska and Padrol~\cite{GonskaPadrol2015}]
Is every (even-dimensional) neighborly polytope inscribable?
\end{question}

Moritz Firsching~\cite{Firsching2015} undertook a quest to find a counterexample. He successfully used non-linear optimization techniques to find polytope realizations. In particular, 
he tried to inscribe the neighborly polytopes enumerated in~\cite{MiyataPadrol2015}. The results are surprising: 
Every neighborly $4$-polytope with $n\leq 11$ vertices, every simplicial neighborly $5$-polytope with $n\leq 10$ vertices, every neighborly $6$-polytope with $n\leq 11$ vertices, and every simplicial neighborly $7$-polytope with $n\leq 11$ vertices is inscribable! Even more,
every simplicial $2$-neighborly $6$-polytope with $n\leq 10$ vertices is also inscribable. This lead Firsching to ask whether the even stronger statement that all $2$-neighborly polytopes are inscribable might be true~\cite[Conj.~1]{Firsching2015}. This may be a very brave conjecture, for which not much evidence is available yet. However, recall that the graph of a $2$-neighborly polytope is complete, and hence has the richest possible structure of Hamiltonian subgraphs~\dots
 
\begin{question}[Firsching~\cite{Firsching2015}]
 Is every (simplicial) $2$-neighborly polytope inscribable?
\end{question}

We can also ask the polar question: What about circumscribability of neighborly polytopes? The results are completely opposite. Chen and Padrol proved that, for any $d\geq 4$, no cyclic $d$-polytope on sufficiently many vertices is circumscribable~\cite{ChenPadrol2015}. (The proof will be sketched below 
in the last section.) They even conjecture that this holds in more generality for neighborly polytopes~\cite[Conj. 7.4]{ChenPadrol2015}. Notice that, since neighborliness can be read on the $f$-vector, this would imply that there is an $f$-vector of a convex polytope that does not belong to any inscribable polytope, namely that of the polar of a neighborly $4$-polytope with sufficiently many vertices. No example of such an $f$-vector has been established so far. Actually, what is known points into the other direction: Every $f$-vector 
of a $3$-dimensional polytope is inscribable~\cite{GonskaZiegler2013}. However, $f$-vectors that are not 
{$k$-scribable} are known for $d\geq 4$ and $1\leq k\leq d-2$ \cite{ChenPadrol2015}.
(A polytope is \emph{$k$-scribable} if it has a realization with all its $k$-faces tangent to the sphere.)

\begin{question}[Gonska and Ziegler~\cite{GonskaZiegler2013}]
 Is there an $f$-vector that is not inscribable?
\end{question}

\section{Universally inscribable}

As we just mentioned, the proof of inscribability for cyclic and many more neighborly polytopes 
given in~\cite{GonskaPadrol2015} still works when we replace the unit ball by any other smooth strictly convex body. 
(This is what Oded Schramm called an \defn{egg} in his celebrated paper
\emph{``How to cage an egg?''}~\cite{Schramm1992}.) Here smoothness is not very important, but strict convexity is. For example, the pigeonhole principle tells us that no simplicial $d$-polytope with more than $d(d+1)$ vertices can be inscribed on the boundary of a $d$-simplex.

Hence many neighborly polytopes (among them all cyclic polytopes) are \defn{universally inscribable}: They can be inscribed into %
any egg. Other examples of such polytopes are the stacked $d$-polytopes arising as a join of a path with a $(d-2)$-simplex; and also Lawrence polytopes~\cite{GonskaPadrol2015}. 

\begin{question}[Gonska and Padrol~\cite{GonskaPadrol2015}]
Which polytopes are inscribable into the boundary of every (smooth) strictly convex body?
\end{question}

An observation of Karim Adiprasito~\cite{GonskaPadrol2015} shows that being inscribable on the sphere is not sufficient for being universally inscribable. (The proof uses projectively unique polytopes~\cite{AdiprasitoZiegler2015}.)
We thank the anonymous reviewer for suggesting the following opposite question. (The reviewer's conjectured answer is yes.)

\begin{question}
Are there polytopes that are inscribable in every egg other than the ellipsoid?
\end{question}

The celebrated Koebe--Andreev--Thurston Theorem states that every $3$-polytope has a realization with all its edges tangent to the sphere. This amazing result has a long history. It seems that it was first proved by Paul Koebe, but only for simple and simplicial polytopes~\cite{Koebe1936}. Bill Thurston later realized~\cite{Thurston1997} that it followed from results of 
Eugene M. Andreev on hyperbolic polyhedra~\cite{Andreev1971} \cite{Andreev1971_2}. Since then, several proofs have been found
 (see~\cite[Section~1.3]{Ziegler2007} and references therein).
Schramm went even further and proved that every $3$-polytope has a realization with all its edges tangent to any given egg~\cite{Schramm1992}.

\begin{theorem}[Schramm~\cite{Schramm1992}]
 For every $3$-polytope $P$, and every smooth strictly convex body $K$, there is a realization $Q$ of $P$ such that each edge of $Q$ is tangent to $K$.
\end{theorem}

There is also the other side of universal inscribability, which would be to focus on the convex bodies. The mathoverflow user called Gregor Samsa asked~\cite{SamsaMathoverflow}:

\begin{quotation}
Is there a convex body in $\RR^d$ such that every combinatorial type of a $d$-dimensional convex polytope can be realized with vertices on its surface?
\end{quotation}
We reproduce the beautiful affirmative answer by Sergei Ivanov~\cite{IvanovMathoverflow}:

\begin{quotation}
Yes, there is such a body. Actually there is one very close to the standard unit ball and containing disjoint representatives of each combinatorial type (but these representatives are very small).

Indeed, every combinatorial type of a $d$-polytope has a realization which looks as follows: there is a ``large'' $(d-1)$-dimensional facet and the remaining surface is a graph over this facet. To construct such a realization, choose a $(d-1)$-facet, pick a hyperplane parallel to it and very close to it (but not intersecting the polytope), and apply a projective map which sends this hyperplane to infinity.

We can further ``flatten'' this realization so that it is very close to its large face. Choose a very small $\varepsilon> 0$, apply a homothety such that the diameter of the polytope becomes less than $\varepsilon$, and place the resulting tiny polytope so that it touches the sphere by a point on its ``large'' face. Then consider the convex hull of the sphere and the polytope. All vertices will be on the boundary of this convex hull if the polytope is sufficiently ``flattened.'' And the convex hull diverges from the ball only in a neighborhood of size $\sim\!\sqrt{\varepsilon}$.

Now pick another combinatorial type of a polytope and repeat the procedure with a much smaller $\varepsilon$ and a location on the sphere chosen so that the neighborhood affected by the second polytope does not interfere with the first one. And so on. Since there are only countably many combinatorial types, they all can be packed into the sphere, provided that $\varepsilon$ goes to $0$ sufficiently fast.
\end{quotation}

\section{Universality}\label{sec:universality}

We move on to a slightly different topic, and a different notion of ``universality.'' Take an inscribable polytope~$P$, and consider the set of all inscribed realizations of~$P$. How does this set look like? To start off, it can be parametrized by the vertex coordinates. Moreover, since any M\"obius transformation that preserves the unit sphere as well as its interior sends an inscribed polytope onto an inscribed polytope, we can safely mod out the action of this group. 
This defines $\rsins(P)$, the \defn{realization space of an inscribed polytope}~$P$ (with $n$ vertices in~$\RR^d$):
\[\rsins(P)=\bigslant{\set{A\in (\Sp{d-1})^n}{\conv(A)\simeq P}}{\Mob{\Sp{d-1}}}.\]

This concept is inspired by an analogue definition for general polytopes.
In that setting, one considers $\rspol(P)$, the set of all realizations of $P$, up to affine 
transformation. From Steinitz's proofs of Steinitz's Theorem \cite{StRa} we know that these sets are relatively
nice when~$P$ is of dimension at most~$3$: They are contractible, contain rational points, etc.; 
see \cite{RichterGebert1997} \cite{RichterGebertZiegler1995}. For higher-dimensional polytopes,
however, the behaviour of realization spaces is much wilder.  This was first observed by Mn\"ev, 
with his celebrated Universality Theorem~\cite{Mnev1986} \cite{Mnev1988} for polytopes and oriented matroids. The polytopal version reads: For every primary basic [open] semi-algebraic set defined over $\Z$ there is a [simplicial] polytope whose realization space is stably equivalent to it. This has many consequences, among them: 
\begin{itemize}
 \item topological: $\rspol(P)$ can have the homotopy type of any arbitrary finite simplicial complex, 
 \item algebraic: there are polytopes that cannot be realized with rational coordinates,
 \item algorithmic: it is ETR-hard to decide if a lattice is the face lattice of a polytope. 
\end{itemize}
Mn\"ev's proof provided polytopes with the desired realization spaces, but could not say anything about their dimensions. Another major step was done later by J\"urgen Richter-Gebert, who proved that there is universality already for $4$-dimensional polytopes~\cite{RichterGebert1997}.

The inscribed picture is similar. Up to dimension~$3$, inscribed realization spaces are reasonable. This follows from the results of Rivin that we have already mentioned \cite{Rivin1994}, which imply that for a $3$-polytope $P$, $\rsins(P)$ is homeomorphic to the polyhedron of angle structures (and hence are contractible). In contrast, results of Adiprasito and Padrol with Louis Theran show that, in arbitrarily high dimensions, there is again universality~\cite{AdiprasitoPadrolTheran2015}. We have not found yet an appropriate notion of stable equivalence for this context, while a universality theorem for general polytopes without stable equivalence is neither available nor in sight.
This forces us to separate the topological, algebraic and algorithmic statements:

\begin{theorem}[Adiprasito, Padrol~\& Theran~\cite{AdiprasitoPadrolTheran2015}]\label{thm:universalityInscribed}\leavevmode
\nopagebreak
\begin{itemize}
 \item For every primary basic semi-algebraic set there is an inscribed polytope whose  realization space is homotopy equivalent to it.
 \item For every finite field extension $F/\QQ$ of the rationals, there is an inscribed polytope that cannot be realized with coordinates in $F$. 
 \item The problem of deciding if a poset is the face lattice of an inscribed [simplicial] polytope
 is polynomially equivalent to the \emph{existential theory of the 
reals (ETR)}. In particular, it is NP-hard.
\end{itemize}
\end{theorem}

In the last point we can even ask the polytopes to be simplicial. This follows from a weak universality theorem for inscribed simplicial polytopes, also from~\cite{AdiprasitoPadrolTheran2015}. In this case, we can only find polytopes whose realization space retracts onto the semi-algebraic set, instead of having homotopy equivalence as in the general case. 

The inscribed analogue to Richter-Gebert's result for $4$-polytopes is still missing.

\begin{question}[Adiprasito, Padrol and Theran~\cite{AdiprasitoPadrolTheran2015}]
Is there universality for inscribed polytopes in bounded dimension, say for $4$-polytopes inscribed into $\Sp{3}$?
\end{question}

The proof of Theorem~\ref{thm:universalityInscribed} strongly relies on the results of Mn\"ev. The strategy is to start with certain polytopes with intricate realization spaces, and then
to show that their inscribed realization spaces are equally involved. In particular, it does not prove that it is hard to decide inscribability once we already know that the face lattice corresponds to a polytope.
However, inscribability is itself a complex condition, and hence one can expect that it increases the complexity
of the corresponding realization spaces. This could lead to a proof of universality that is intrinsic to inscribed polytopes, and hopefuly to advances in the previous question. 
A first step in this direction could be to find a polytope~$P$ such that $\rsins(P)$ is disconnected while $\rspol(P)$ is not.
\begin{question}
Is there universality for the realization spaces $\rsins(P)$ 
for a class of inscribable polytopes~$P$ whose general realization spaces $\rspol(P)$ are trivial
(in particular, contractible)?
\end{question}

\section{(\emph{i,j})-scribability}

We have already referred to the Koebe--Andreev--Thurston Theorem. One is tempted
to ask if similar behaviors might also appear in higher dimensions. The answer is no.
Egon Schulte used an inductive argument over non-inscribable/non-circumscribable $3$-polytopes to show that, for every $d\geq 4$, and every $0\leq k\leq d-1$, 
there are $d$-polytopes that cannot be realized with all their $k$-faces tangent to the sphere \cite{Schulte1987}. Of course, this opens the door to ask for a characterization of $k$-scribable polytopes. There are almost no results in this direction, which is definitely an interesting research topic.

We will, however, consider a different problem. If we take an edge-scribed realization of a $3$-polytope, it has also the following property: All the vertices are outside the ball while all the facets cut the ball. This kind of realizations are studied in~\cite{ChenPadrol2015}. Here a polytope is said to be \defn{$(i,j)$-scribed} if all its $i$-faces ``avoid'' the sphere while all the $j$-faces ``cut'' it.
The definitions for cutting and avoiding that seem to work better say that a face $F$ \defn{cuts} the ball if there is a point of the unit ball in its relative interior, and that it \defn{avoids} the ball
if there is a hyperplane $H$, supporting for~$F$, which completely contains $P$ and the ball in one of the closed halfspaces it defines. (This somehow involved definition is needed if we want a self-dual concept that reduces to classical $k$-scribability when $i=j=k$.)

The weakest condition occurs when $i=0$ and $j=d-1$:
\begin{question}
Does every polytope have a realization where every vertex avoids the ball and every
facet cuts the ball?
\end{question}

The answer is most probably no, but so far no proof has been given (although Karim Adiprasito
has suggested that one should be able to obtain counterexamples from glueing some of the large
projectively unique polytopes 
from \cite{AdiprasitoZiegler2015}). 
There are some results in the opposite direction: For $d=3$, 
the edge-scribed realization is also $(0,d-1)$-scribed.
Every inscribable 
polytope has directly a $(0,d-1)$-scribed realization too. So, in particular, cyclic polytopes and many (all?) neighborly polytopes have such realizations.
And obviously, circumscribable polytopes also have such a realization. This includes stacked polytopes, which are always
circumscribable~\cite{ChenPadrol2015}.

What about other values of $i$ and $j$? The best offenders found in \cite{ChenPadrol2015} are even-dimensional cyclic polytopes (and their duals):

\begin{theorem}[Chen and Padrol \cite{ChenPadrol2015}]
    If an even-dimensional cyclic polytope has sufficiently many vertices, then it is not $(1,d-1)$-scribable (which in particular implies that it is not circumscribable). 
\end{theorem}

\begin{proof}[Proof sketch]
The first step is to associate to each vertex 
of a $(1,d-1)$-scribed cyclic polytope 
the spherical cap consisting of the points of $\SS^{d-1}$ visible from 
it. 
This yields a configuration of spherical caps on~$\SS^{d-1}$. These are said to form a \defn{$k$-ply} system
if no point of $\SS^{d-1}$ is contained in the interior of more than $k$ caps.

Now the Sphere Separator Theorem 
of Gary Miller, Shang-Hua Teng, Bill Thurston and Stephen Vavasis~\cite{MillerTengThurstonVavasis1997}
states that the intersection graph of
a $k$-ply system on $\SS^{d-1}$ has a separator of size $O(k^{{1}/{(d-1)}}n^{1-{1}/({d-1})})$.
The proof is astonishingly simple and beautiful (cf. \cite[Thm. 8.5]{PachAgarwal1995}): With 
a M\"obius transformation, we can assume that the origin is a center-point 
of the centers of the caps. 
The next step is to compute the probability
that a random linear hyperplane intersects a cap, which depends only 
on the area covered by the cap. Since our system is $k$-ply, we can estimate
the sum of these volumes because we know the surface area of the sphere. 
This is used to show that a random linear hyperplane hits very few caps, 
whose removal separates the graph.

So how do we tie this in with cyclic polytopes? The key observation is that a set of
points induce a $k$-ply system if and only if the convex hull of every $k$-set 
intersects the sphere. A \defn{$k$-set} is a subset of $k$ points that can be
separated from the others with a hyperplane. Even-dimensional cyclic polytopes
have a lot of nice properties, among them oriented matroid rigidity. This allows us
to show that every $k$-set with $k\geq \frac{3}{2}d - 1$ contains a facet. If the 
realization was $(0,d-1)$-scribed, this facet would intersect the sphere, and hence
the cyclic polytope induces a $k$-ply set. But in this case, the intersection
graph consists of the edges of the polytope avoiding the sphere. If all the edges
avoided the sphere, this would be a complete graph, which obviously does not have
a small separator.%
\end{proof}

\subsubsection*{Acknowledgements}
We are grateful to Karim Adiprasito, Hao Chen, Igor Rivin, Juanjo Ru\'e, Francisco Santos, and the anonymous 
referee for very valuable discussions, suggestions, and references; and to Sergei Ivanov for 
letting us reproduce his mathoverflow answer.

\bibliographystyle{spmpsci}
\bibliography{Roundness}

\begin{thebibliography}{10}
\providecommand{\url}[1]{{#1}}
\providecommand{\urlprefix}{URL }
\expandafter\ifx\csname urlstyle\endcsname\relax
  \providecommand{\doi}[1]{DOI~\discretionary{}{}{}#1}\else
  \providecommand{\doi}{DOI~\discretionary{}{}{}\begingroup
  \urlstyle{rm}\Url}\fi

\bibitem{AdiprasitoPadrolTheran2015}
Adiprasito, K., Padrol, A., Theran, L.: Universality theorems for inscribed
  polytopes and {D}elaunay triangulations.
\newblock Discrete Comput. Geom. \textbf{54}, 412--431 (2015)

\bibitem{AdiprasitoZiegler2015}
Adiprasito, K., Ziegler, G.M.: Many projectively unique polytopes.
\newblock Inventiones Math. \textbf{199}, 581--652 (2015)

\bibitem{Alon1986}
Alon, N.: The number of polytopes, configurations and real matroids.
\newblock Mathematika \textbf{33}(1), 62--71 (1986)

\bibitem{Andreev1971}
{Andreev}, E.M.: {On convex polyhedra in {L}oba\v{c}evsk\i\u{\i} spaces}.
\newblock {Math. USSR, Sb.} \textbf{10}, 413--440 (1971)

\bibitem{Andreev1971_2}
{Andreev}, E.M.: {On convex polyhedra of finite volume in
  {L}oba\v{c}evsk\i\u{\i} space}.
\newblock {Math. USSR, Sb.} \textbf{12}, 255--259 (1971)

\bibitem{BenderGaoRichmond1992}
Bender, E.A., Gao, Z.C., Richmond, L.B.: Submaps of maps. {I}. general 0–1
  laws.
\newblock J. Combinat. Theory, Ser.~B \textbf{55}(1), 104--117 (1992)

\bibitem{Caratheodory1911}
Carath{\'e}odory, C.: {{\"U}ber den Variabilit{\"a}tsbereich der
  {Fourier}'schen Konstanten von positiven harmonischen Funktionen}.
\newblock Rendiconto del Circolo Matematico di Palermo \textbf{32}, 193--217
  (1911)

\bibitem{ChenPadrol2015}
Chen, H., Padrol, A.: Scribability problems for polytopes.
\newblock Preprint, August 2015, 24~pp.,
  \href{http://arxiv.org/abs/arXiv:1508.03537}{\url{arXiv:1508.03537}}

\bibitem{DancigerMaloniSchlenker2015}
Danciger, J., Maloni, S., Schlenker, J.M.: Polyhedra inscribed in a quadric.
\newblock Preprint, Oct.~2014, 42~pp.,
  \href{http://arxiv.org/abs/arXiv:1410.3774}{\url{arXiv:1410.3774}}

\bibitem{Dillencourt1990}
Dillencourt, M.B.: Toughness and {D}elaunay triangulations.
\newblock Discrete Comput. Geom. \textbf{5}(6), 575--601 (1990)

\bibitem{DillencourtSmith1995}
Dillencourt, M.B., Smith, W.D.: {A linear-time algorithm for testing the
  inscribability of trivalent polyhedra}.
\newblock Internat. J. Comput. Geom. Appl. \textbf{5}(1-2), 21--36 (1995).
\newblock Eighth Annual ACM Symposium on Computational Geometry (Berlin, 1992)

\bibitem{DillencourtSmith1996}
Dillencourt, M.B., Smith, W.D.: {Graph-theoretical conditions for
  inscribability and Delaunay realizability}.
\newblock Discrete Mathematics \textbf{161}(1--3), 63--77 (1996)

\bibitem{Firsching2015}
Firsching, M.: {Realizability and inscribability for some simplicial spheres
  and matroid polytopes}.
\newblock Preprint, October 2015, 25~pp.,
  \href{http://arxiv.org/abs/1508.02531}{\url{arXiv:1508.02531}}

\bibitem{Gale1963}
Gale, D.: Neighborly and cyclic polytopes.
\newblock In: Proc. {S}ympos. {P}ure {M}ath., {V}ol. {VII}, pp. 225--232. Amer.
  Math. Soc., Providence, R.I. (1963)

\bibitem{GonskaPadrol2015}
Gonska, B., Padrol, A.: {Neighborly inscribed polytopes and Delaunay
  triangulations}.
\newblock Advances in Geometry  (in press).
\newblock Preprint
  \href{http://arxiv.org/abs/1308.5798v2}{\url{arXiv:1308.5798}}

\bibitem{GonskaZiegler2013}
Gonska, B., Ziegler, G.M.: {Inscribable stacked polytopes}.
\newblock Advances in Geometry \textbf{13}(4), 723--740 (2013)

\bibitem{Gruenbaum2003}
Gr{\"u}nbaum, B.: {Convex polytopes}, \emph{{Graduate Texts in Mathematics}},
  vol. 221, second edn.
\newblock Springer-Verlag, New York (2003)

\bibitem{GruenbaumJucovic1974}
Gr{\"u}nbaum, B., Jucovi{\v{c}}, E.: On non-inscribable polytopes.
\newblock Czechoslovak Math. J. \textbf{24(99)}, 424--429 (1974)

\bibitem{HodgsonRivinSmith1992}
Hodgson, C.D., Rivin, I., Smith, W.D.: {A characterization of convex hyperbolic
  polyhedra and of convex polyhedra inscribed in the sphere}.
\newblock Bull. Amer. Math. Soc. \textbf{27}(2), 246--251 (1992)

\bibitem{IvanovMathoverflow}
Ivanov, S.: {Can all convex polytopes be realized with vertices on surface of
  convex body?}
\newblock MathOverflow,
  \href{http://mathoverflow.net/q/107113}{\url{mathoverflow.net/q/107113}},
  Sept.~2012

\bibitem{Koebe1936}
Koebe, P.: {K}ontaktprobleme der konformen {A}bbildung.
\newblock Berichte Verh. S\"achs. Akademie der Wissenschaften Leipzig,
  Math.-Phys. Klasse \textbf{88}, 141--164 (1936)

\bibitem{McMullen1970}
McMullen, P.: {The maximum numbers of faces of a convex polytope}.
\newblock Mathematika \textbf{17}, 179--184 (1970)

\bibitem{MillerTengThurstonVavasis1997}
Miller, G.L., Teng, S.H., Thurston, W., Vavasis, S.A.: {Separators for
  sphere-packings and nearest neighbor graphs}.
\newblock J. ACM \textbf{44}(1), 1--29 (1997)

\bibitem{MiyataPadrol2015}
Miyata, H., Padrol, A.: Enumerating of neighborly polytopes and oriented
  matroids.
\newblock Experimental Math. \textbf{24}(4), 489--505 (2015)

\bibitem{Mnev1986}
Mn{\"e}v, N.E.: {The topology of configuration varieties and convex polytopes
  varieties}.
\newblock Ph.D. thesis, St. Petersburg State University, St. Petersburg, RU
  (1986).
\newblock 116 pages,
  \href{http://www.pdmi.ras.ru/~mnev/mnev\_phd1.pdf}{\url{pdmi.ras.ru/~mnev/mnev_phd1.pdf}}

\bibitem{Mnev1988}
Mn{\"e}v, N.E.: {The universality theorems on the classification problem of
  configuration varieties and convex polytopes varieties}.
\newblock In: {Topology and Geometry---{R}ohlin {S}eminar}, \emph{{Lecture
  Notes in Math.}}, vol. 1346, pp. 527--544. Springer-Verlag, Berlin Heidelberg
  (1988)

\bibitem{Motzkin1957}
Motzkin, T.S.: Comonotone curves and polyhedra.
\newblock Bulletin American Mathematical Society \textbf{63}, 35 (1957).
\newblock Abstract

\bibitem{PachAgarwal1995}
Pach, J., Agarwal, P.K.: Combinatorial Geometry.
\newblock Wiley-Interscience Series in Discrete Mathematics and Optimization.
  John Wiley \& Sons, Inc., New York (1995)

\bibitem{Padrol2013}
Padrol, A.: {Many neighborly polytopes and oriented matroids}.
\newblock {Discrete Comput. Geom.} \textbf{50}(4), 865--902 (2013)

\bibitem{RichterGebert1997}
Richter-Gebert, J.: {Realization spaces of polytopes}, \emph{{Lecture Notes in
  Mathematics}}, vol. 1643.
\newblock Springer-Verlag, Berlin (1996)

\bibitem{RichterGebertZiegler1995}
Richter-Gebert, J., Ziegler, G.M.: Realization spaces of $4$-polytopes are
  universal.
\newblock Bulletin Amer. Math. Soc. \textbf{32}, 403--412 (1995)

\bibitem{Rivin1993}
Rivin, I.: On geometry of convex ideal polyhedra in hyperbolic {$3$}-space.
\newblock Topology \textbf{32}(1), 87--92 (1993)

\bibitem{Rivin1994}
Rivin, I.: {Euclidean structures on simplicial surfaces and hyperbolic volume}.
\newblock {Annals Math. (2)} \textbf{139}(3), 553--580 (1994)

\bibitem{Rivin1996}
Rivin, I.: {A characterization of ideal polyhedra in hyperbolic {$3$}-space}.
\newblock Annals Math. (2) \textbf{143}(1), 51--70 (1996)

\bibitem{Rivin2003}
Rivin, I.: {Combinatorial optimization in geometry}.
\newblock {Advances Applied Math.} \textbf{31}(1), 242--271 (2003)

\bibitem{SamsaMathoverflow}
Samsa, G.: Can all convex polytopes be realized with vertices on surface of
  convex body?
\newblock MathOverflow.
\newblock
  \href{http://mathoverflow.net/q/107096}{\url{mathoverflow.net/q/107096}},
  Sept. 2012

\bibitem{Schramm1992}
Schramm, O.: {How to cage an egg}.
\newblock Inventiones Math. \textbf{107}(3), 543--560 (1992)

\bibitem{Schulte1987}
Schulte, E.: {Analogues of {S}teinitz's theorem about non-inscribable
  polytopes}.
\newblock In: {Proc.\ ``Intuitive Geometry'' ({S}i{\'o}fok, 1985)},
  \emph{{Colloq. Math. Soc. J{\'a}nos Bolyai}}, vol.~48, pp. 503--516.
  North-Holland, Amsterdam (1987)

\bibitem{Seidel1991}
Seidel, R.: {Exact upper bounds for the number of faces in {$d$}-dimensional
  {V}orono\u\i\ diagrams}.
\newblock In: {Applied Geometry and Discrete Mathematics}, \emph{{DIMACS Ser.
  Discrete Math. Theoret. Comput. Sci.}}, vol.~4, pp. 517--529. Amer. Math.
  Soc., Providence, RI (1991)

\bibitem{Smith1991}
Smith, W.D.: On the enumeration of inscribable graphs.
\newblock Manuscript, 7~pages, NEC Research Institute, 1991;
  \href{http://citeseer.ist.psu.edu/viewdoc/summary?doi=10.1.1.29.4543}{\url{doi:10.1.1.29.4543}}

\bibitem{Steiner1832}
Steiner, J.: {S}ystematische {E}ntwicklung der {A}bh\"angigkeit geometrischer
  {G}estalten von einander.
\newblock Fincke, Berlin (1832).
\newblock Also in: Gesammelte Werke, Vol.~1, Reimer, Berlin 1881, pp. 229--458

\bibitem{Steinitz1928}
{Steinitz}, E.: {{\"U}ber isoperimetrische Probleme bei konvexen Polyedern}.
\newblock {J. Reine Angewandte Math.} \textbf{159}, 133--143 (1928)

\bibitem{StRa}
Steinitz, E., Rademacher, H.: Vorlesungen \"uber die Theorie der Polyeder.
\newblock Springer-Verlag, Berlin (1934).
\newblock Reprint, Springer-Verlag 1976

\bibitem{Thurston1997}
Thurston, W.P.: Geometry and {T}opology of $3$-{M}anifolds.
\newblock Lecture Notes, Princeton University, Princeton 1977--1978;
  \url{http://library.msri.org/books/gt3m/}

\bibitem{Ziegler1995}
Ziegler, G.M.: Lectures on {P}olytopes, \emph{Graduate Texts in Math.}, vol.
  152.
\newblock Springer, New York (1995)

\bibitem{Ziegler2007}
Ziegler, G.M.: {Convex polytopes: extremal constructions and {$f$}-vector
  shapes}.
\newblock In: {Geometric Combinatorics}, \emph{{IAS/Park City Math. Ser.}},
  vol.~13, pp. 617--691. Amer. Math. Soc., Providence, RI (2007)

\end{thebibliography}

\end{document}